\documentclass[a4paper, reqno, 12pt]{amsart}

\usepackage{geometry}             

\usepackage{float}
\usepackage{bm}
\usepackage{fullpage,xcolor}
\usepackage[mathscr]{euscript}
\usepackage[all]{xy}
\usepackage{epsfig}
\usepackage{amsfonts}
\usepackage{mathptmx}

\usepackage{graphicx}
 \usepackage{amssymb} 
\usepackage{amsmath}
\usepackage{amsthm}
\usepackage{mathrsfs}
\usepackage{epstopdf}
\usepackage{url}
\usepackage[msc-links,alphabetic]{amsrefs}
\usepackage{tikz}

\usepackage{color}
\makeatletter

\@addtoreset{equation}{section}
\makeatother 

\theoremstyle{plain}
\newtheorem{theorem}{Theorem}[section]
\newtheorem{corollary}[theorem]{Corollary}
\newtheorem{proposition}[theorem]{Proposition}

\theoremstyle{definition}
\newtheorem{definition}[theorem]{Definition}
\newtheorem{example}[theorem]{Example}
\newtheorem{examples}[theorem]{Examples}
\newtheorem{remark}[theorem]{Remark}
\newtheorem{note}[theorem]{Note}

\usepackage{latexsym}
 
\title[Directional Invariants of Doubly Periodic Tangles]
  {Directional Invariants of Doubly Periodic Tangles}

\author{Ioannis Diamantis}
\address{Department of Data Analytics and Digitalisation,
Maastricht University, School of Business and Economics,
P.O.Box 616, 6200 MD, Maastricht,
The Netherlands.}
\email{i.diamantis@maastrichtuniversity.nl}

\author{Sofia Lambropoulou}
\address{School of Applied Mathematical and Physical Sciences, National Technical University of Athens, Zografou campus, GR-15780 Athens, Greece.}
\email{sofia@math.ntua.gr}
\urladdr{http://www.math.ntua.gr/~sofia}

\author{Sonia Mahmoudi}
\address{Advanced Institute for Materials Research, Tohoku University, 2-1-1 Katahira, Aoba-ku, Sendai 980-8577, Japan;  RIKEN iTHEMS, Wako, Saitama 351-0198, Japan}
\email{sonia.mahmoudi@tohoku.ac.jp}



\begin{document}

\setcounter{section}{-1}

\begin{abstract} 
In this paper we define novel topological invariants of doubly periodic tangles (DP tangles). DP tangles are embeddings of curves in the thickened plane with translational symmetries in two independent directions. We first organize the components of a DP tangle into different {\it interlinked compounds}, which are invariants of a DP tangle. The notion of interlinked compound leads to the classification of DP tangles according to their {\it directional type}. We then prove that the directional type is an invariant of DP tangles using the concept of {\it axis-motif}, which can be viewed as the blueprint of a DP tangle.
\end{abstract}

\maketitle


\section{Introduction}\label{sec:0}

A doubly periodic tangle (DP tangle) is an embedding of curves in the thickened plane $\mathbb{E}^2 \times I$ which is symmetric under translations in two transverse directions. So, DP tangles can be viewed as lifts of links in the thickened torus, $T^2 \times I$. An example can be viewed in Figure~\ref{DPtangle}. DP tangles serve as a significant framework for analyzing and understanding the topological properties of interwoven filament systems across micro-, meso- and macro-scales, including, but not limited to, polymer melts  \cite{Eleni1, Eleni2, Eleni3}, fabric-like structures \cite{Sabetta, Sonia1}, molecular chemistry \cite{Yaghi, Treacy}, and cosmic filaments \cite{Bond,Hong1,Hong2}. 

\smallbreak

The topological classification of DP tangles is at least as hard a problem as the full classification of knots and links in the three-space. Such problems are approached by constructing topological invariants. We do that by first representing a DP tangle by its quotient, that we call {\it motif}, under a {\it periodic lattice}, which serves as a frame for the double translational symmetry of the DP tangle.  This approach reduces the complexity of the classification problem. In this direction, several numerical, polynomial and finite type invariants have been constructed for doubly periodic textile structures,such as woven and knitted fabrics (works initiated by Grishanov et al., cf. \cite{Grishanov1, Grishanov.part1, Grishanov.part2, Morton, Grishanov.Vassiliev1, Grishanov.Vassiliev2, Kurlin, Eleni4}), which form a particular subclass of DP tangles. The equivalence relation that these invariants respect is based on assumptions of minimal motifs, that is motifs that are minimal for reproducing the DP tangles under 2-periodic boundary conditions. Note however that obtaining a minimal motif of a DP tangle is known to be a non-trivial problem \cite{Grishanov1}. 

\smallbreak

Motivated by the classification problem of DP tangles and the above works, we first established in \cite{DLM} the topological equivalence of DP tangles on the level of arbitrary motifs of theirs, without the assumption of minimality (see Theorem~\ref{th:equivalence}). 

In this paper we construct some new topological invariants of DP tangles. More precisely, for a DP tangle $\tau_{\infty}$ we start by considering a {\it flat motif} $\tau$, that is, the diagram of a motif on the flat torus. We first define the notion of an {\it interlinked compound} of $\tau$, which is a subtangle of $\tau$ that can be split from other components. Using this concept, we organize interlinked compounds into classes (null-homotopic, ribbon and cover) and subclasses, that are invariant under the equivalence relation of DP tangles (see Section~\ref{sec:compounds}). We further define the {\it class of a motif} $\tau$, by ordering the classes of its interlinked compounds, and prove that it is a topological invariant of the DP tangle $\tau_{\infty}$ (see Theorem~\ref{th:interlinked-compounds}). 

We next move to defining directional invariants of DP tangles,  generalizing the notion of `axial type' of a textile strand in \cite{Morton}. Thanks to the notion of interlinked compound, we can extend the axial type in the general setting of DP tangles. Indeed, we introduce the notion of {\it directional element} of a motif $\tau$, defined as being either a component (an isolated knot or an essential component) or an interlinked compound (a chain-link or a full-polycatenane compound), see Definition~\ref{def:element}. Then, in Theorem~\ref{th:distinct-directions} we prove that the number of directions in a motif is a topological invariant of the DP tangle. 

The above lead, in turn, to the notion of {\it axis-motif}, which can be viewed as the blueprint of a DP tangle (Definition~\ref{def:axis-motif}). In particular, the axis-motif $\alpha(\tau)$ of a motif $\tau$  captures the number of elements in $\tau$ of a given direction, and consequently the {\it directional type} of the corresponding DP tangle, which is related to the class of the motif, and which constitutes another invariant of DP tangles (see Theorem~\ref{th:directional types}). 

All the above invariants of DP tangles are measures that naturally inform on their topological complexity, they refer to global topological properties of a DP tangle, and they add to the list of the existing invariants. In the end of the paper we compute our invariants on several DP tangles, comparing them at the same time with some known numerical invariants. 

\smallbreak

The paper is organized as follows: 
In \S~\ref{sec:setup} we present the setting and the results from \cite{DLM}. More precisely, we first recall the definitions of DP tangles and motifs and we state the generalized Reidemeister theorem. Then, in \S~\ref{sec:compounds} we introduce the notion of interlinked compounds, which allows us to assign a class to each DP tangle. This leads to our first invariants of DP tangles. In \S~\ref{sec:type}, and using the concept of interlinked compounds, we define the notion of direction of an element of a DP tangle. This notion leads to a classification of DP tangles based on their directional type, that we also introduce. Finally, in \S~\ref{sec:tables}, we provide examples highlighting the importance of the invariants presented in this paper (see Tables~\ref{table1} and \ref{table2}).

\section{Preliminaries on equivalence of DP tangles}\label{sec:setup}

In this section we recall the notions of DP tangles, their generating motifs and their equivalence, as established in \cite{DLM}, where we formulate and prove DP tangle equivalence for arbitrary DP tangles and arbitrary motifs of theirs. This study is based on works initiated by Grishanov et al., cf. \cite{Grishanov1, Grishanov.part1, Grishanov.part2, Morton}. It should be pointed out that these prior works focus on a subclass of DP tangles related to textiles (such as woven and knitted fabrics), and the results therein are based on assumptions of minimal motifs. 
In DP textiles, their motifs consist of only essential simple closed curves embedded in the thickened torus. One of the reasons is that the weaving and knitting machines that make these textiles cannot build null-homotopic components. In this study we consider the full theoretical generality.

\subsection{DP tangles and motifs}

Due to the double translational symmetry of {\it DP tangles}, a DP tangle is defined as the lift to $\mathbb{E}^2 \times I$ of a link $\tau$ in the thickened torus $T^2 \times I$, where $\mathbb{E}^2$ stands for the Euclidean plane and $I=[0,1]$ the unit interval. View Figure~\ref{DPtangle}  for an illustrated example. More precisely, consider a basis $B=\{u,v\}$ of $\mathbb{E}^2$ and let $\rho: \mathbb{E}^2 \rightarrow{} T^2$ be the covering map that assigns a meridian $m$ of $T^2$ to $v$ and a longitude $l$ of $T^2$ to $u$. This covering map $\rho$ extends to a covering map $\tilde{\rho}$: $ \mathbb{E}^2 \times I \rightarrow{} T^2 \times I$. 

\begin{figure}[ht]
\centerline{\includegraphics[width=5in]{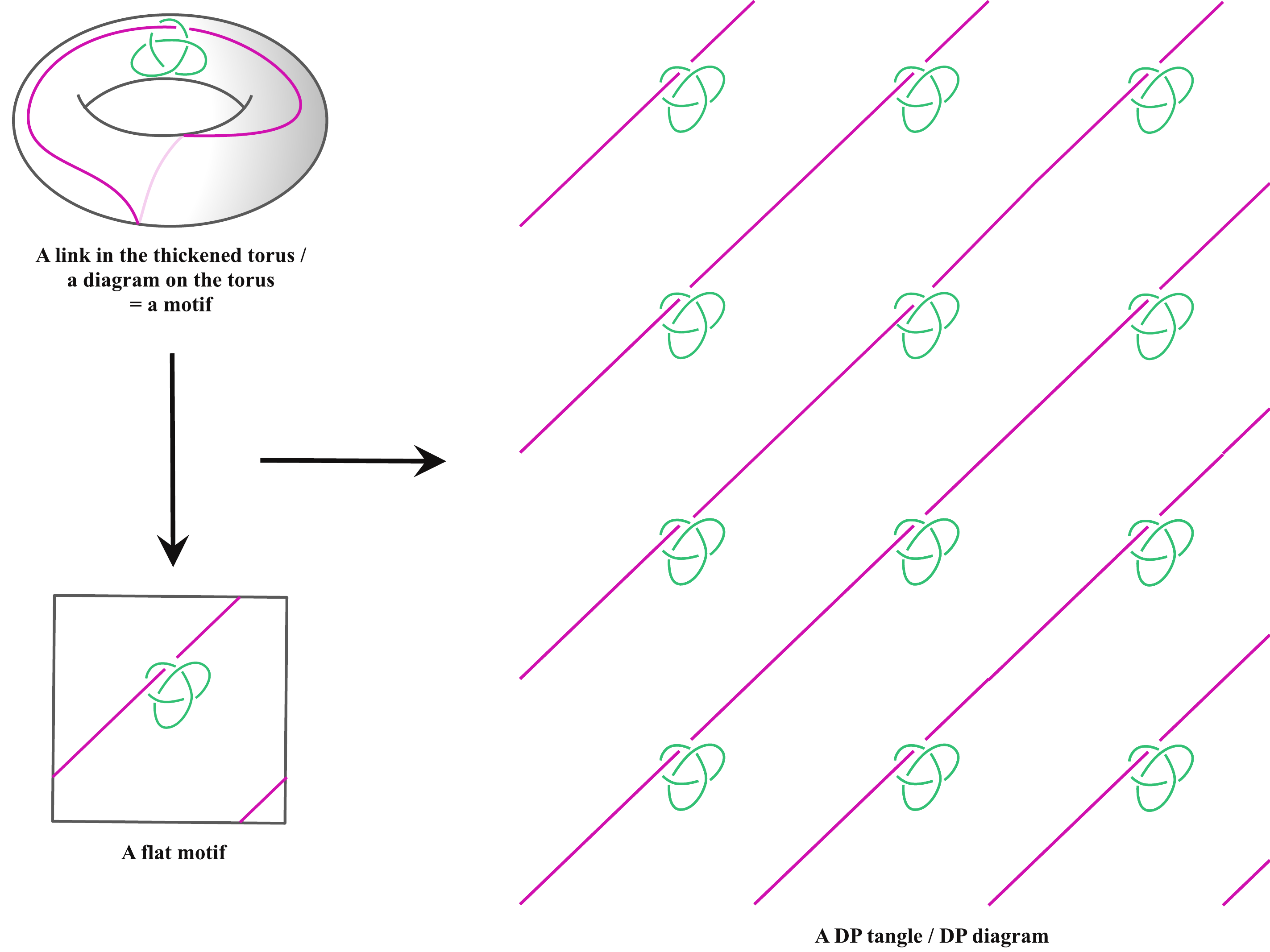}}
\vspace*{8pt} 
\caption{\label{DPtangle} 
A  DP tangle as the lift of a link in $T^2 \times I$, and a corresponding flat motif. }
\end{figure}

\begin{definition}\label{def:DP tangle}
A {\it DP tangle}, denoted $\tau_{\infty}$, is the lift under $\tilde{\rho}$ of a link $\tau$ in $T^2 \times I$. Moreover, a {\it DP diagram}, denoted  $d_{\infty}$, is the lift under $\rho$ of a diagram $d$ of $\tau$ in $T^2 \times \{0\}$. We call $d$ (resp. $\tau$) a {\it motif} of $d_{\infty}$ (resp. $\tau_{\infty}$).
\end{definition}

Assigning further a basis $B=\{u,v\}$ of $\mathbb{E}^2$ to a DP tangle together with a choice of {\it longitude} $l$ and {\it meridian} $m$ for $T^2$, the set of points $\Lambda (u,v) = \{xu + yv\, |\, x,y \in \mathbb{Z}\}$ generated by $B$ defines a {\it periodic lattice} isomorphic to $\mathbb{Z}^2$. Moreover, the tangle diagram contained in the flat torus $T^2$ that arises as the identification space by identifying the opposite sides of the boundary of any one of the parallelograms of $\Lambda$, represents a motif $d$ in $T^2$ and is called a {\it flat motif}, also denoted $d$ and we shall write $d = d_{\infty} / \Lambda$. In particular, it is well-known that the same periodic lattice $\Lambda = \Lambda (u,v) = \Lambda' (u',v')$ can be generated by two different bases $B=\{u,v\}$ and $B'=\{u',v'\}$ if and only if for $x_1, x_2, x_3, x_4 \in \mathbb{Z}$,
$$\begin{pmatrix}
u' \\
v' 
\end{pmatrix}
=
\begin{pmatrix}
x_1 & x_2 \\
x_3 & x_4 
\end{pmatrix}
\cdot
\begin{pmatrix}
u \\
v 
\end{pmatrix}
{\it  , where }
\mid{x_1 x_4 - x_2 x_3}\mid = \pm 1.$$

\begin{figure}[H]
\centerline{\includegraphics[width=5in]{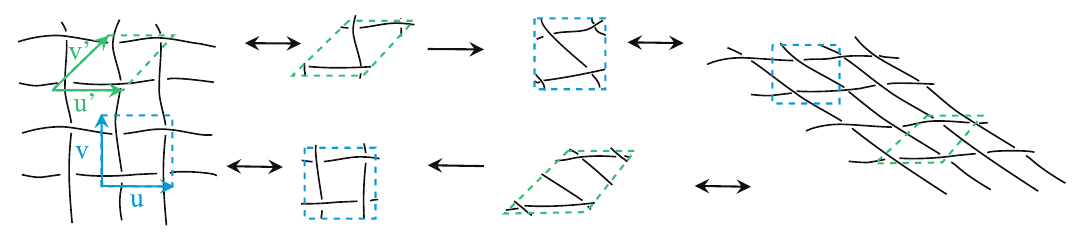}}
\vspace*{8pt}
\caption{\label{shearing} 
A fixed lattice for a DP tangle, with two different bases of $\mathbb{E}^2$ (green and blue), which give rise to a shearing of the DP tangle.}
\end{figure}

It is important to highlight that given a flat motif $d$ of a DP tangle $\tau_{\infty}$ associated to a lattice $\Lambda$, then any finite cover of $d$ also defines a flat motif of $\tau_{\infty}$ (see  Figure~\ref{Tknot-Tlink} for an example). We thus recall the definition of minimal lattice associated to a DP tangle.

\begin{figure}[H]
\centerline{\includegraphics[width=5.1in]{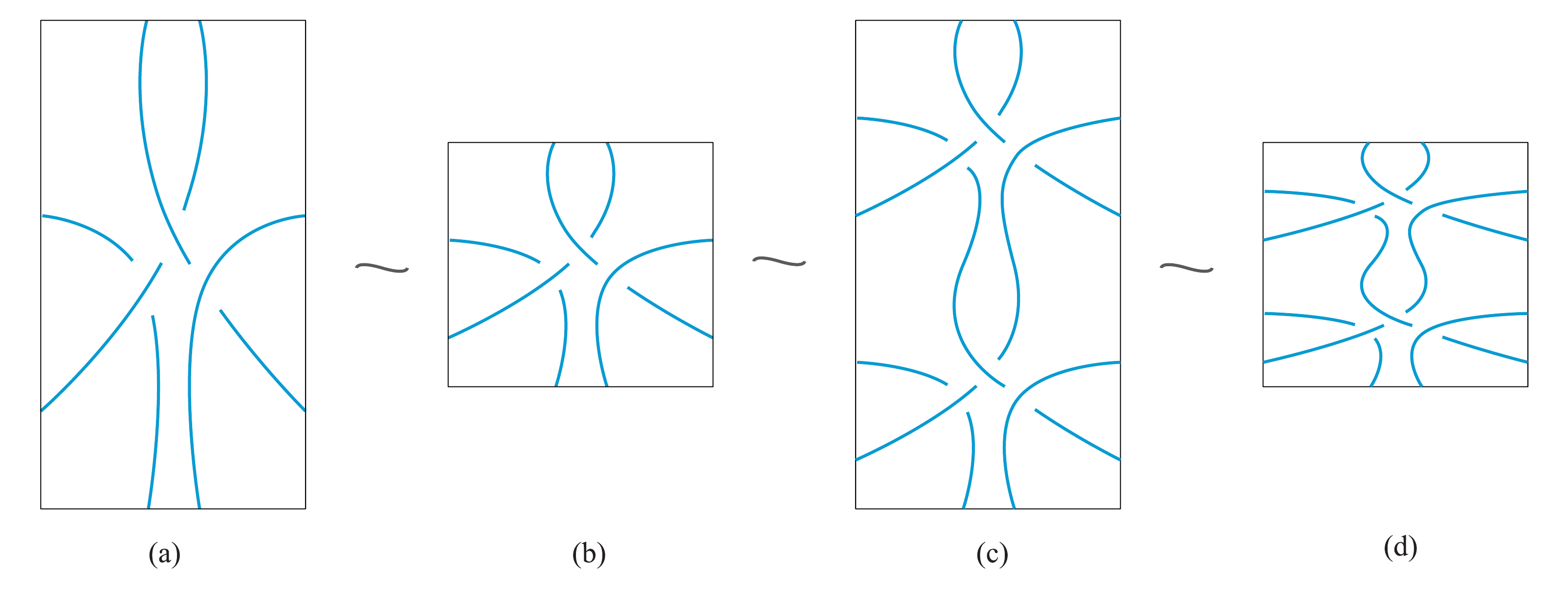}}
\vspace*{8pt}
\caption{\label{Tknot-Tlink} Motif (a) is a stretching of (b); motifs (b) and (c) are scale equivalent; motif (d) is a contraction of (c).}
\end{figure}

\begin{definition}\label{def:minimal lattice}
A {\it minimal lattice}, denoted by $\Lambda_{min}$, of a DP tangle $\tau_{\infty}$, is defined as a lattice satisfying $\Lambda \subseteq \Lambda_{min}$, in the set of all periodic lattices $\Lambda$ of $d_{\infty}$, up to any area-preserving transformation of $\mathbb{E}^2$. Moreover, the (flat) motif $d_{min} = d_{\infty} / \Lambda_{min}$ is called a {\it minimal (flat) motif of} $d_{\infty}$. 
\end{definition}

\subsection{Equivalence of DP tangles and their motifs}

We first recall the definition of equivalence for DP tangles that preserves the double periodicity.

\begin{definition}\label{def:DPequivalence}
Two DP tangles (resp. DP diagrams) are {\it equivalent} if they differ by doubly periodic (diagrammatic) isotopies and orientation preserving invertible affine transformations of the plane $\mathbb{E}^2$ that carry along the DP tangles (resp. DP diagrams). 
\end{definition}

In \cite{DLM}, we studied the equivalence of DP tangles as reflected in their flat  motifs. More precisely, doubly periodic isotopies include local motif isotopies (that is, surface isotopies and Reidemeister moves), as well as global isotopies, induced by invertible affine transformations of the plane. These  consist of both non-area preserving transformations such as re-scalings (stretches, contractions), as well as area preserving transformations like rigid translations and rotations of the plane, or shear deformations. Re-scalings correspond to torus isotopies, like inflation or contractions, while shear deformations correspond to {\it Dehn twists} of the underlying torus, which are orientation preserving self-homeomorphisms of the torus, as related in \cite{Grishanov1,Grishanov.part1} (see Figure~\ref{Rtwists} for an example). More precisely,  consider first the lattice $\Lambda$, which is generated by  two bases $B$ and $B'$ as defined above. We now assign the longitude $l$ of $T^2$ to $u'$ and the meridian $m$ to $v'$ of $B'$. The corresponding covering map, say $\rho': \mathbb{E}^2 \rightarrow{} T^2$, generates a new motif $d'$, which differs from the motif $d$ associated to the basis $B$ by a finite sequence of Dehn twists. 
Two (flat) motifs related by a finite sequence of Dehn twists are said to be {\it Dehn equivalent}. Finally, in the DP tangle isotopy one should also include static changes, such as shifts or re-scalings of the underlying lattice. Re-scalings of the underlying lattice correspond to distinct finite covers of the same motif and we shall be referring to them as {\it scale equivalence}. More precisely, let $\Lambda_0$, $\Lambda_1$ and $\Lambda_2$ be three (not necessarily distinct) periodic lattices associated to a DP diagram $d_{\infty}$ such that $\Lambda_1 \subseteq \Lambda_0$ and $\Lambda_2 \subseteq \Lambda_0$. Let also $d_0 = d_{\infty} / \Lambda_0$, $d_1 = d_{\infty} / \Lambda_1$ and $d_2 = d_{\infty} / \Lambda_2$ be three (flat) motifs of $d_{\infty}$. Then, $d_1$ and $d_2$ arise as finite covers of $d_0$ and are said to be {\it scale equivalent}.

\begin{figure}[H]
\centerline{\includegraphics[width=4.5in]{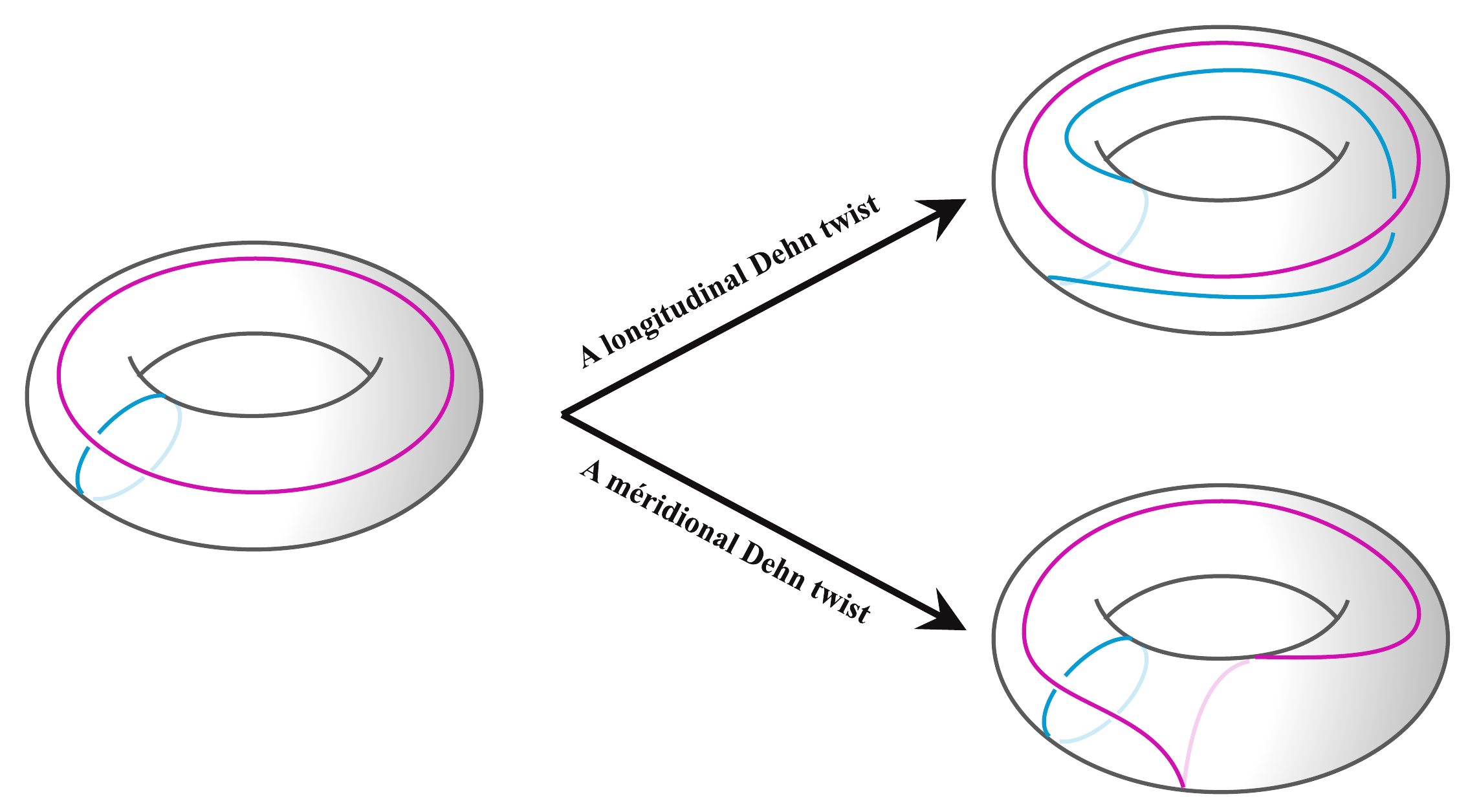}}
\vspace*{8pt}
\caption{\label{Rtwists} 
Two Dehn twist of the torus: a longitudinal (purple) and a meridional one (blue)}.
\end{figure}

The above lead to the following theorem for the DP tangle equivalence, which is very important for defining topological invariants of DP tangles.

\begin{theorem}[\cite{DLM}]\label{th:equivalence}
Let $\tau_{1,\infty}$ and $\tau_{2,\infty}$ be two DP tangles in $\mathbb{E}^2 \times I$, with corresponding DP diagrams $d_{1,\infty}$ and $d_{2,\infty}$. Let also $\Lambda_1$ and $\Lambda_2$ be the supporting point lattices such that $d_i = d_{i,\infty} / \Lambda_i$ is a flat motif of $d_{i,\infty}$ for $i \in \{1,2\}$. 
Then $\tau_{1,\infty}$ and $\tau_{2,\infty}$ are {\it equivalent} if and only if $d_1$ and $d_2$ are related by a finite sequence of shifts, motif isotopy moves, Dehn twists, orientation preserving affine transformations and scale equivalence.
\end{theorem}

\begin{figure}[H]
\centerline{\includegraphics[width=5in]{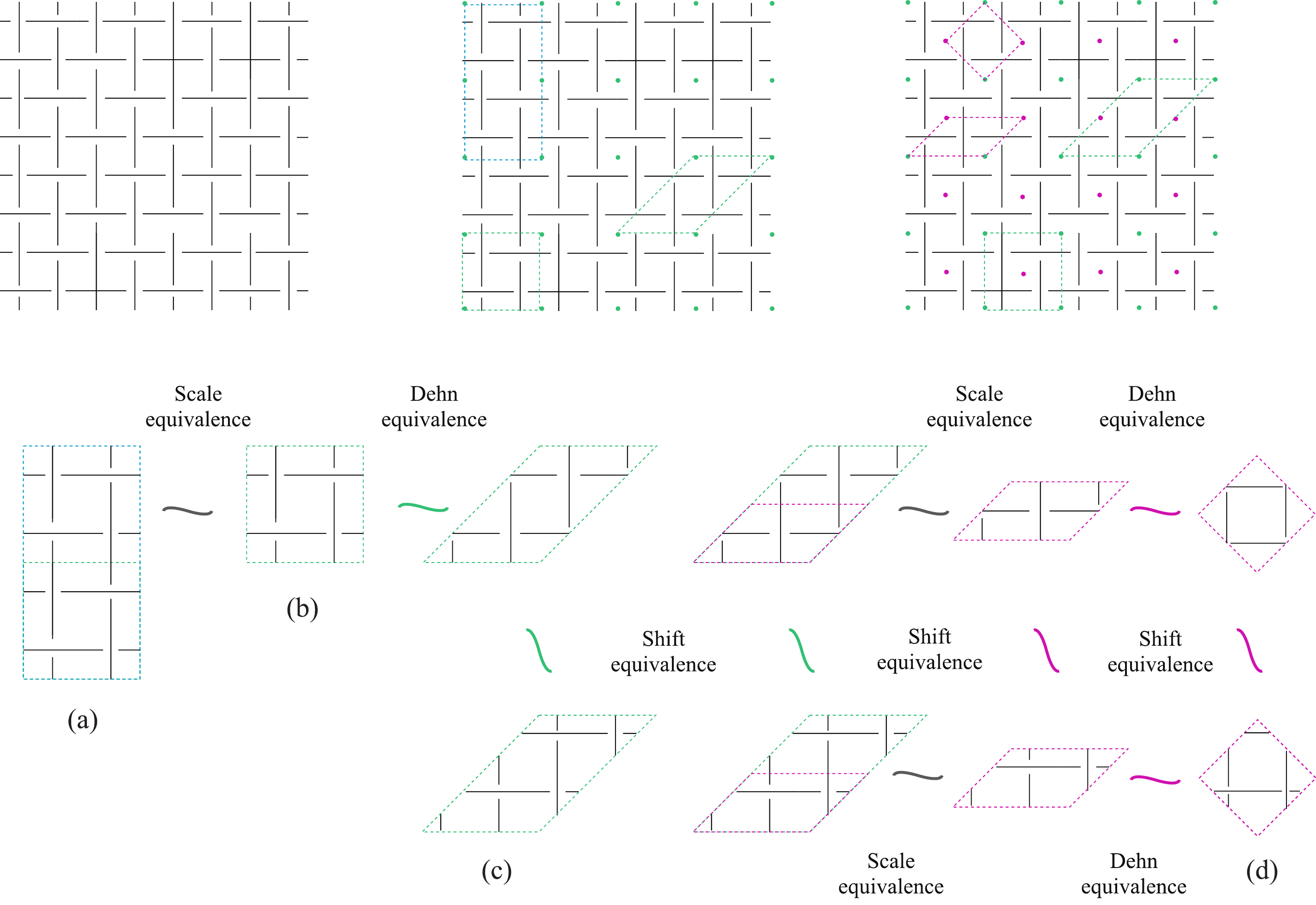}}
\vspace*{8pt}
\caption{\label{crossing} Equivalence transformations among different generating motifs for the same DP tangle, culminating to the minimal motif (d). Same colours of frames indicate the same lattices, related by area-preserving transformations. }
\end{figure}

\section{Interlinked compounds and the class of a motif} \label{sec:compounds}

In this section we introduce the notion of {\it interlinked compound} of a motif $\tau_{\infty}$ and we organize the interlinked compounds of a motif into invariant subclasses. Using this notion, we then distinguish motifs according to their interlinked compounds and we prove that the class of a motif $\tau_{\infty}$ is a topological invariant of $\tau_{\infty}$ (see Theorem~\ref{th:interlinked-compounds}).

\subsection{The class of an interlinked compound}

We recall that for a link embedded in the torus surface $T^2$ each component is a closed curve, which is either {\it null-homotopic}, namely contractible to a point in $T^2$, or {\it essential}, namely non-contractible in $T^2$. The null-homotopic components can be isotoped to circles, while the essential components comprise torus knots and torus links. According to our convention, an $(a,b)${\it -torus knot} is an embedding of $S^1$ in $T^2$ that winds $a$ times along the longitude  and $b$ times along the meridian of the torus, where $a$ and $b$ are coprime integers. On the level of the corresponding flat motif, the $(a,b)$-torus knot intersects $a$ times  the meridian and $b$ times  the longitude of the torus. 
The lift of an $(a,b)$-torus knot to the universal cover $\mathbb{E}^2$ is an infinite set of parallel lines of direction $(a,b)$. An $(a',b')${\it -torus link} is a link in $T^2$ consisting of $g$ components, each being an $(a,b)$-torus knot, where $g=gcd(a',b')$, $a'= g \cdot a$ and $b'=g \cdot b$.

\smallbreak

Let now $\tau_\infty$ be a DP tangle arising as the lift of a motif $\tau$ embedded in the thickened torus $T^2 \times I$ to $\mathbb{E}^2 \times I$. Unlike links in $T^2$, $\tau$ may contain crossings, therefore $\tau$ will consist in finitely many (closed) components, possibly knotted and possibly linked. 

\smallbreak

In the case of links forming a motif $\tau$, we are particularly interested in (sub)sets of interlinked components. Namely, we define: 

\begin{definition} \label{def:compound}
    Let $\tau$ be a motif in $T^2 \times I$ and $\mu$ be a set of closed components of $\tau$ such that,
    \begin{enumerate}
   \item[i.] $\mu$ forms a split sublink of $\tau$, that is, a sublink that can split from the rest of $\tau$ by isotopy; 
    \smallbreak  
    \item[ii.] $\mu$ cannot split into separate sublinks under isotopy.
\end{enumerate}
    Then, $\mu$ is said to be an {\it interlinked compound} of $\tau$. 
    The definition carries through, analogously, to the flat motif of $\tau$ and to the DP tangle $\tau_\infty$. 
\end{definition}

\noindent Note that Definition~\ref{def:compound} allows for an interlinked compound to be just a single component of the motif. 
In Figure~\ref{interlinked-compounds}, the motifs (a) and (e) consist of two interlinked components, the motifs (b) and (h) consist of three interlinked components, while the motifs (c), (d), (f) and (g) consist of  only one interlinked compound. Similarly, in Figure~\ref{DPtangle} we have one interlinked compound in the motif and the flat motif, which lifts to an infinitum of identical interlinked compounds in $\tau_\infty$.

\begin{figure}[ht]
\begin{center}
\includegraphics[width=5in]{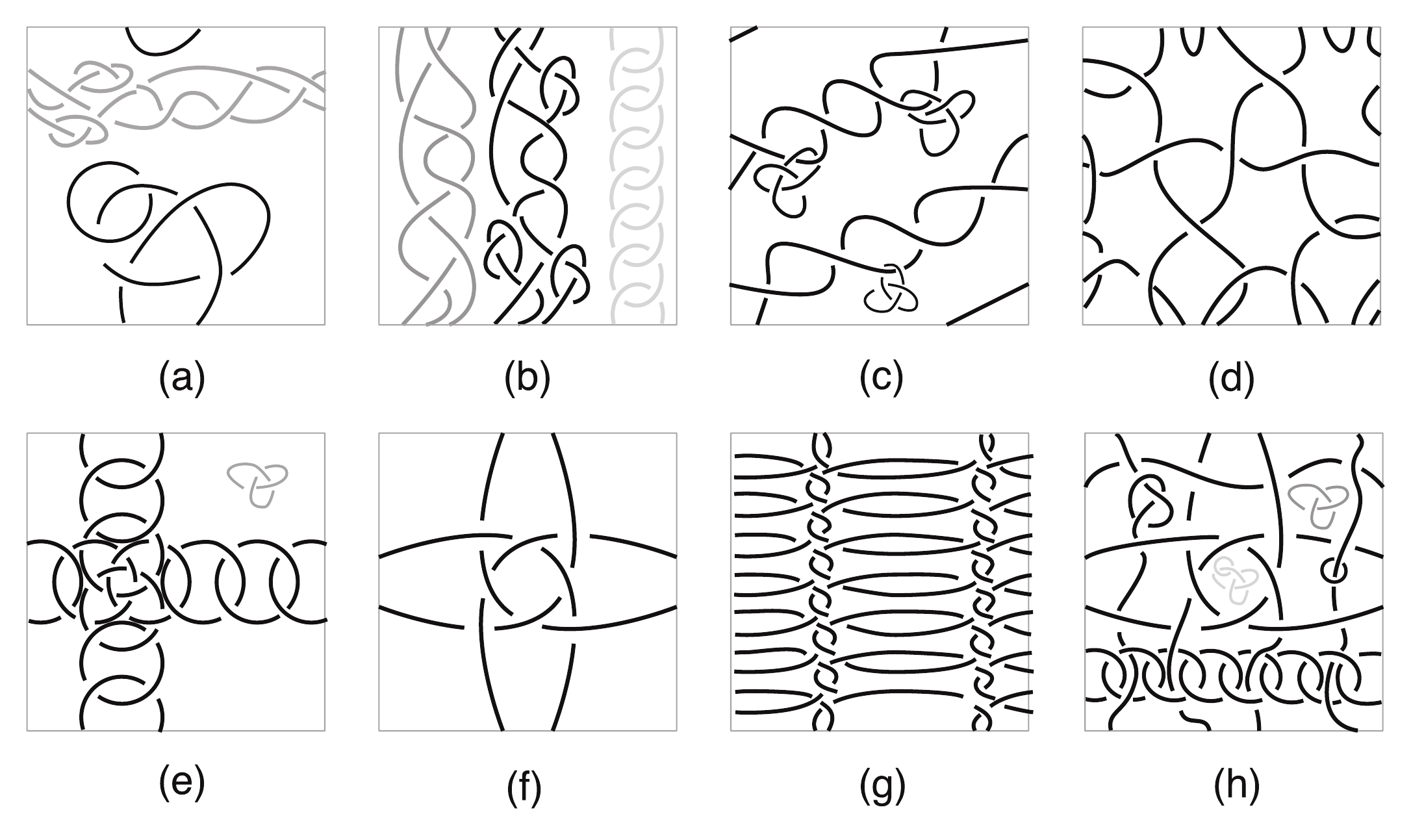}
\end{center}
\caption{\label{interlinked-compounds} Examples of flat motifs.  Different shades help identify the different interlinked compounds in the motifs.}
\end{figure}

A direct consequence of Definition~\ref{def:compound} and the types of embedded closed curves in $T^2$ is the following distinction of interlinked compounds into different classes.

\begin{figure}[ht]
\begin{center}
\includegraphics[width=5in]{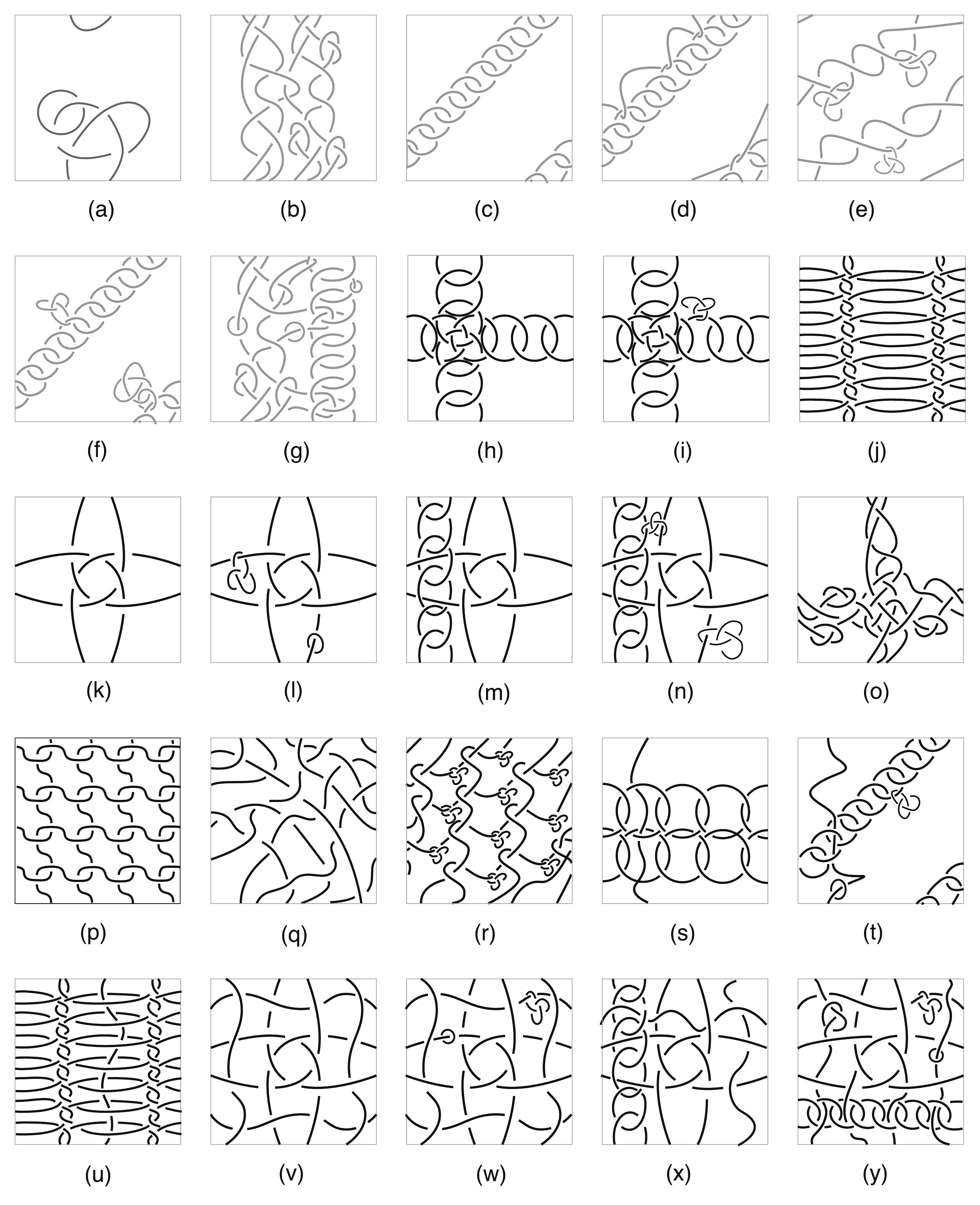}
\end{center}
\caption{\label{DPmotifs} Examples of flat motifs consisting of a single interlinked compound. Motif (a) is a null-homotopic compound, motifs (b)--(g) are ribbon compounds, and motifs (h)--(y) are cover compounds.}
\end{figure}

\begin{definition}\label{def:typescompound} \rm 
 An interlinked compound $\mu$ of a (flat) motif $\tau$ or a DP tangle $\tau_\infty$  belongs to one of the following classes: 
\begin{enumerate}
   \item[a.]  $\mu$ can be enclosed, up to motif isotopy, in a 3-ball in $T^2 \times I$. Then it is referred to as {\it null-homotopic compound} of $\tau$ and it lifts in $\mathbb{E}^2 \times I$ to a disjoint union of null-homotopic compounds.
    \smallbreak  
    
   \item[b.]  $\mu$ can  be enclosed, up to motif isotopy, in an essential thickened ribbon in $T^2 \times I$. In this case $\mu$ is referred to as {\it ribbon compound} of $\tau$ and it lifts in $\mathbb{E}^2 \times I$ to an infinite disjoint union of identical ribbon compounds in $\tau_\infty$. 
   \smallbreak  
   
   \item[c.] $\mu$ cannot be enclosed in either a 3-ball or a thickened ribbon in $T^2 \times I$, under any motif isotopy.  In this case $\mu$  is referred to as {\it cover compound} of $\tau$ and it lifts in $\mathbb{E}^2 \times I$ to a single cover compound in $\tau_\infty$. 
  \end{enumerate}
\end{definition} 

\noindent  Figure~\ref{DPmotifs} (a) shows an example of a null-homotopic compound. Figure~\ref{DPmotifs} (b) -- (g) are examples of ribbon compounds, while Figure~\ref{DPmotifs} (h) -- (y) are examples of cover compounds.

\begin{proposition}\label{classcompound} 
 The class of an interlinked compound as null-homotopic, ribbon or  cover is an invariant under motif equivalence resp.  DP tangle equivalence.
\end{proposition} 

\begin{proof} Let  $\mu$ be an interlinked compound of a (flat) motif $\tau$ or a DP tangle $\tau_\infty$. Recall from Theorem~\ref{th:equivalence} that DP tangle equivalence is generated by shifts, motif isotopy moves, Dehn twists, orientation preserving affine transformations and scale equivalence on the level of (flat) motifs.  We claim that the interlinked compound in the resulting motif under any of the these moves will be of the same  class as $\mu$. Indeed, under the effect of shifts, motif isotopy moves or orientation preserving affine transformations, the validity of the statement is straightforward. 

Let us consider the effect of Dehn twists. A Dehn twist does not alter the number of components or the number of crossings or the linking, and thus it preserves the class of an interlinked compound. More precisely, if $\mu$ is a null-homotopic compound in $T^2 \times I$ it will remain null-homotopic after a Dehn twist. If $\mu$ is a ribbon compound, it will have non empty intersections with the torus meridian or the torus longitude or both. Then, up to Dehn twists, $\mu$ can be arranged to intersect either the torus meridian or the torus longitude, and so remains a ribbon compound.  If $\mu$ is a cover compound, it will have non empty intersections with both the torus meridian and the torus longitude and this will not change under any  Dehn twist, so the result remains a cover compound.  

Let us, finally, consider the effect of scale equivalence. Note that scale equivalence may or may not change the number of components or interlinked compounds in a motif. Consider for example a flat motif $\tau$ containing only a single $(1,0)$-torus knot component (see Figure~\ref{scale-direction}, on the left), thus one interlinked compound $\mu$. Then, a scale equivalent flat motif $\tau'$  made of two copies of  $\tau$ would contain: either a single $(1,0)$-torus knot component, if it is a left-right copy, or  a $(2,0)$-torus link component if it is a top-down copy. In the first case $\tau$ and $\tau'$ both contain a single interlinked compound $\mu'$. In the second case, however, the  $(1,0)$-torus knot duplicates in $\tau'$. So,  $\tau'$ contains two disjoint identical interlinked compounds, either one of which belong to the same class as the original $\mu$. As another example, for a null-homotopic compound $\mu$, scale equivalence will increase the number of copies $\mu$, while a cover compound will remain a cover compound. In general, scale equivalence may increase  the number of copies of an interlinked compound $\mu$, but it will preserve the class of $\mu$ in any motif or in the universal cover $\mathbb{E}^2 \times I$ of $T^2 \times I$. Therefore DP equivalence does not alter the class of an interlinked compound.
\end{proof}

\begin{definition} \label{order-compounds} 
A cover compound is said to be {\it stronger} than a ribbon compound and both are said to be {\it stronger} than a null-homotopic compound. 
\end{definition}

\begin{corollary} \rm \label{classordering} 
The ordering {\it (cover,  ribbon,  null-homotopic)} is well-defined in the set of equivalence classes of interlinked-compounds under motif equivalence resp. DP tangle equivalence. 
\end{corollary}

The corollary is an immediate consequence of Proposition~\ref{classcompound}.

\subsection{Subclasses of interlinked compounds}\label{subsec:special-cases}

As {\it basic interlinked compounds} we consider the following: {\bf  null-homotopic compound, essential ribbon compound, chain-link ribbon compound,   full-polycatenane compound} (see descriptions below). All  other interlinked compounds can be created as interlinkages of the basic types. From all possibilities of the structure of an interlinked compound we will single out and name some cases of interest. In the lists below, the basic interlinked compounds are marked in bold face. The namings will be in accordance to the stronger class of compound involved in the interlinkage.  

Let $\mu$ be an interlinked compound of a motif $\tau$ resp. of the DP tangle $\tau_\infty$. We distinguish the cases:

\bigbreak

\noindent {\it Ribbon compounds:}  If $\mu$ is a ribbon compound then we have the following possibilities:

\smallbreak

\begin{itemize}
    \item  $\mu$ is an {\it {\bf essential ribbon compound} } if it consists of only essential closed curves in $T^2 \times I$. An example is demonstrated by each one of the three  interlinked compounds of motif (b) in Figure~\ref{DPmotifs}.  
\smallbreak

    \item  $\mu$ is a { \it {\bf  chain-link ribbon compound}} if it consists of an interlinked combination of only null-homotopic components. An example is the interlinked compound in Figure~\ref{DPmotifs}(c). 
\smallbreak

    \item  $\mu$ is a {\it chain-essential ribbon compound} if it consists of essential curves interlinked with chain-links (see for example Figure~\ref{DPmotifs}(d)).
\smallbreak

    \item  $\mu$ is a {\it null-essential ribbon compound} if it consists of an essential ribbon compound interlinked with a null-homotopic compound (see for example  Figure~\ref{DPmotifs}(e)).
\smallbreak

    \item  $\mu$ is a {\it null-chain ribbon compound} if it consists of a chain-link ribbon compound interlinked with a null-homotopic compound (see for example  Figure~\ref{DPmotifs}(f)).
\smallbreak

    \item  if $\mu$ is a combination of all the above, we shall be refer to as  {\it mixed ribbon compound} (for an illustration see Figure~\ref{DPmotifs}(g)). 
\end{itemize}

\bigbreak

\noindent \noindent {\it Cover compounds:} If $\mu$ is a cover compound then we have the following possibilities:

\smallbreak

\begin{itemize}
    \item  $\mu$ is a {\it polycatenane compound} if it consists of only interlinked null-homotopic components. In particular, polycatenane compounds can be distinguished as follows:
\smallbreak

    \begin{itemize}
      \item [*] $\mu$ is a {\it chain-polycatenane compound} if it can be defined as the interlinkage of at least two chain-link ribbon compounds. An example is illustrated in Figure~\ref{DPmotifs}(h),
      \smallbreak
      
      \item [*] $\mu$ is a {\it null-chain-polycatenane compound} if it consists of a chain-polycatenane compound interlinked with at least one  null-homotopic compound, as illustrated in Figure~\ref{DPmotifs}(i),
      \smallbreak
      
      \item [*]  $\mu$ is a {\it {\bf full-polycatenane compound }} if it consists of  null-homotopic components only,  without null-homotopic compounds, as the examples illustrated in Figure~\ref{DPmotifs}(j) and~(k),  
      \smallbreak
      
      \item [*]  $\mu$ is a {\it null-full-polycatenane compound} if it consists of a full-polycatenane compound interlinked with at least one null-homotopic compound, as the example illustrated in Figure~\ref{DPmotifs}(l),
      \smallbreak
      
      \item [*]  $\mu$ is a {\it chain-full-polycatenane compound} if it consists of a full-polycatenane compound interlinked with at least one chain-link compound, as the examples illustrated in Figure~\ref{DPmotifs}(m),
      \smallbreak
      
      \item [*] $\mu$ is a {\it null-chain-full-polycatenane compound} if it is a combination of all the above classes of polycatenanes, as the example illustrated in Figure~\ref{DPmotifs}(n),
    \end{itemize}
    \smallbreak

    \item  if $\mu$ contains essential closed curves, then we have the following possibilities:
    \smallbreak

    \begin{itemize}
      \item [*] $\mu$ is an {\it essential cover compound} if it consists in an interlinked combination of (only) essential ribbon compounds. This case is exemplified by the interlinked compound of motifs (o), (p) and (q) in Figure~\ref{DPmotifs}. 
      \smallbreak
      
      \item [*]  $\mu$ is a {\it null-essential cover compound} if it consists of essential cover compounds interlinked with null-homotopic compounds (see for example Figure~\ref{DPmotifs}(r)),
      \smallbreak
      
      \item [*]  $\mu$ is a {\it chain-essential cover compound} if it consists of closed curves interlinked with chain-links (see for example Figure~\ref{DPmotifs}(s)),
      \smallbreak
      
      \item [*]  $\mu$ is a {\it null-chain-essential cover compound} if it consists of essential closed curves interlinked with chain-links and null-homotopic compounds (see for example Figure~\ref{DPmotifs}(t)),
      \smallbreak
      
      \item [*]   $\mu$ is a {\it essential-full-polycatenane cover compound} if it consists of essential curves interlinked with a full-polycatenane (see examples in Figure~\ref{DPmotifs}(u) and~((v)).
      \smallbreak
      
      \item [*]  $\mu$ is a {\it null-essential-full-polycatenane cover compound} if it consists of essential curves interlinked with a full-polycatenane and with null-homotopic compounds (see for example Figure~\ref{DPmotifs}(w)),
      \smallbreak
      
      \item [*]   $\mu$ is a {\it chain-essential-full-polycatenane cover compound} if it consists of essential curves interlinked with chain links and with a full-polycatenane (see for example Figure~\ref{DPmotifs}(x)),
      \smallbreak
      
      \item [*]  if $\mu$ is a combination of all the above, we shall be refer to as {\it mixed cover compound} (see for example  Figure~\ref{DPmotifs}(y)).  
      \end{itemize}
      
\end{itemize} 

\noindent  By an adaptation of the proof of Proposition~\ref{classcompound} it is clear that the subclass of an interlinked compound does not change under DP tangle equivalence. Therefore we have the following result.
 
\begin{theorem} \label{th:subclasses} 
Let $\mu$ be an interlinked compound of a motif $\tau$ of a DP tangle $\tau_\infty$. The subclass of $\mu$ as listed above is a topological invariant of $\mu$. 
\end{theorem}

\subsection{The class of a motif} 

Following Definition~\ref{def:typescompound}, we now define the class of a motif as follows.

\begin{definition} \label{motif-compound} 
Let  $\tau$ be a motif of a DP tangle $\tau_\infty$. Then $\tau$ is said to be of the class {\it cover motif} if it contains only one or more cover compounds. It is said to be of the class {\it null-cover motif} if it consists of at least one cover compound and at least one null-homotopic compound. $\tau$ is said to be of the class {\it ribbon-cover motif} if it consists of at least one cover compound and at least one ribbon compound. It is said to be of the class {\it null-ribbon-cover motif} if it consists of at least one cover compound, as well as at least one ribbon compound, and at least one null-homotopic compound. Further, $\tau$ is said to be of the class {\it ribbon motif} if it contains only ribbon compounds. Moreover, $\tau$ is said to be of the class {\it null-ribbon motif} if it contains at least a ribbon compound and a null-homotopic compound. Finally, $\tau$ is said to be of the class {\it null-homotopic motif} if it only consists of null-homotopic compounds.  The above categories exhaust the possible {\it classes} of motifs and their respective DP tangles. 
\end{definition}

Proposition~\ref{classcompound}, Corollary~\ref{classordering} and Theorem~\ref{th:subclasses} lead now to the following culminating result.

\begin{theorem} \label{th:interlinked-compounds} 
Let  $\tau$ be a motif of  a DP tangle $\tau_\infty$.  The set of all subclasses of the interlinked compounds of $\tau$  resp. $\tau_\infty$ is a topological invariant of $\tau$  resp. $\tau_\infty$. 
 Furthermore, the class of  $\tau$ resp. of  $\tau_\infty$  as cover, null-cover, ribbon-cover, null-ribbon-cover, ribbon, null-ribbon, or null-homotopic  is a topological invariant of $\tau$  resp. $\tau_\infty$. 
\end{theorem}

\section{Directional invariants and axis-motifs of DP tangles} \label{sec:type}

In this section we introduce different types of DP tangles based on the {\it directions} of elementary constituents of theirs. Our intention is to construct invariants of directional type of DP tangles that will be defined using the notion of their axis-motifs. The definitions, discussions and results in this section about motifs and their corresponding DP tangles extend equally to their underlying diagrams, flat motifs and DP diagrams.

\subsection{The notion of directions}\label{sec:direction}

To define the notion of direction of a motif, as well as of its interlinked compounds, we start this subsection by introducing the more refined notion of a `directional element' of a motif. We start with the following definitions:

\begin{definition}\label{def:isolated}
 Let $\tau$ be a (flat) motif. 
 An {\it isolated knot} in $\tau$ is a component of a null-homotopic compound. This compound could be an interlinked compound by itself or may be interlinked with basic compounds of other stronger classes in $\tau$, according to Definition~\ref{order-compounds} (recall Subsection~\ref{subsec:special-cases}). 
\end{definition}
\noindent For example, in motif (n) of Figure 7, the upper small trefoil is an isolated knot interlinked with a ribbon and a cover compound. 

\begin{definition}\label{def:element}
 Let $\tau$ be a (flat) motif. 
 A {\it directional element} or just {\it element} of $\tau$ can be an isolated knot, an essential component, a chain-link compound, belonging to a ribbon or cover compound, or a full-polycatenane compound.
\end{definition}

\begin{example}
We provide examples of elements in the motifs of Figure~\ref{DPmotifs}:
\begin{itemize}
      \item [-] The motifs (a), (l), (n), (t), (w) and (y) contains two isolated knots, while motifs (e) and (f) contain three isolated knots.
      \smallbreak
      
      \item [-]  The motifs (b) and (o) contain five essential components, while motifs (p), (v), (w) contain four essential components and motif (q) six essential components.
      \smallbreak
      
      \item [-] The motifs (c), (d), (f), (g), (m), (n), (t), (x), (y) contain one chain link while the motifs (h) and (s) contain two chain links.
      \smallbreak
      
      \item [-] The motifs (j)-(n) and (u)-(y) contain a full polycatenane. 
\end{itemize}    
\end{example}

From Definitions~\ref{def:element}, \ref{def:typescompound} and the discussion on the Subsection~\ref{subsec:special-cases}, we can define the notion of direction for an element.

\begin{definition}\label{def:direction-element}
 The {\it direction of an element} is defined as follows. Let $e$ be an element of a (flat) motif $\tau$. Then:

 \begin{enumerate}
   \item[a.]  if $e$ is an {\it isolated knot}, thus homotopic to the trivial knot, then $e$ is said to have direction $(0,0)$.
   \smallbreak  
   
   \item[b.] if $e$ is an {\it essential component}, homotopic to an $(a,b)$-torus knot, then $e$ is said to have direction $(a,b)$. 
   \smallbreak   
   
   \item[c.] if $e$ is a {\it chain-lin}k, whose centerline curve is isotopic to an $(a,b)$-torus knot, then $e$ is said to have direction $(a,b)$.
   \smallbreak 
   
   \item[d.] if $e$ is a {\it full polycatenane}, then $e$ is said to have direction $(\infty,\infty)$. 
  \end{enumerate} 
\end{definition}

\begin{remark}
Two elements $e$ and $e'$ with respective directions $(a,b)$ and $(-a,-b)$ or $(-a,b)$ and $(a,-b)$ are said to have the {\it same direction}. 
    
Note that an $(a,b)$-torus knot/link and a $(b,a)$-torus knot/link differ only by the identification order of the opposite sides of a parallelogram, so by the exchange of meridian and longitude. In our theory we consider motifs in a (thickened) torus, so the (longitude, meridian) pair corresponds to the basis ($u$, $v$) of the Euclidean plane. Therefore, the pairs $(a,b)$ and $(b,a)$ are not considered to have the same direction.
\end{remark}

Using the notion of direction for an element we can define the notion of direction of an interlinked compound, which includes the cases where an element and an interlinked compound are identified.

\begin{definition}\label{def:direction-compound} 
The {\it direction of an interlinked compound} is defined as follows. Let $\mu$ be an interlinked compound of a (flat) motif $\tau$ or a DP tangle $\tau_\infty$. Then:
 \begin{enumerate}
   \item[N.]  if $\mu$ is a {\it null-homotopic compound} of $\tau$, then $\mu$ is said to have direction $(0,0)$.
    \smallbreak  
    
   \item[R.] if $\mu$ is a {\it ribbon compound} of $\tau$, whose centerline curve is isotopic to an $(a,b)$-torus knot, 
   then $\mu$ is said to have direction $(a,b)$. 
   \smallbreak  
   
   \item[C.] if $\mu$ is a {\it cover compound} of $\tau$, then we distinguish the following cases:
   \smallbreak
   \begin{enumerate}
       \item [1.] if $\mu$ is an {\it essential cover compound}, a {\it chain-polycatenane compound} or a {\it chain-essential cover compound}, then $\mu$ is an interlinkage of a finite set of distinct ribbon compounds $\{R_1, \ldots, R_n\}$, such that for all $i \in \{1, \ldots, n\}, R_i$ has direction $(a_i,b_i)$. Then $\mu$ is said to have direction $\{(a_1,b_1), \ldots, (a_n,b_n)\}$.
       \smallbreak
       
       \item [2.] if $\mu$ is a {\it null-essential cover compound} or {\it null-chain cover compound} or {\it null-chain-essential cover compound}, then $\mu$ satisfies conditions (N) and (1) and is thus said to have direction $\{(0,0), (a_1,b_1), \ldots, (a_n,b_n)\}$.
       \smallbreak
       
       \item [3.] if $\mu$ is a {\it full polycatenane}, then $\mu$ is said to have direction $(\infty,\infty)$.
       \smallbreak
       
       \item [4.] if $\mu$ is a {\it null-full-polycatenane}, then $\mu$ satisfies conditions (N) and (3) and is thus said to have direction $\{(\infty,\infty), (0,0)\}$.
       \smallbreak
       
       \item [5.] if $\mu$ is a {\it chain-full-polycatenane} or a {\it essential-full cover compound} or a {\it chain-essential-full cover compound}, then $\mu$ satisfies conditions (1) and (3) and is thus said to have direction $\{(\infty,\infty), (a_1,b_1), \ldots, (a_n,b_n)\}$.
       \smallbreak
       
       \item [6.] if $\mu$ is a {\it null-chain-full cover compound}, a {\it null-essential-full cover compound} or a {\it mixed cover compound}, then $\mu$ satisfies conditions (N), (1) and (3), and is thus said to have direction $\{(\infty,\infty), (0,0), (a_1,b_1), \ldots, (a_n,b_n)\}$.
   \end{enumerate}
  \end{enumerate}     
\end{definition}

Examples are given in the next subsection (see Examples~\ref{ex:fig-DP tangle}). 

\smallbreak
The notion of direction of a (flat) motif resp. a DP tangle follows now directly from Definition~\ref{def:direction-element}: 

\begin{definition}\label{def:direction}
The {\it direction of a (flat)  motif} resp. the {\it direction of a DP tangle} is defined as follows. Let $\tau$ be a (flat) motif of a DP tangle $\tau_\infty$. We define the direction of $\tau$ to be the set of directions of its elements, taken up to homotopy in the flat motif. Furthermore, the direction of $\tau_\infty$ is defined to be the same as the direction of $\tau$.  
\end{definition}

\begin{remark}\label{rem:classencoding}
The direction of a (flat)  motif (resp. a DP tangle) encodes the class of the motif (resp. the DP tangle), as defined in Definition~\ref{motif-compound}. 
\end{remark}

\begin{remark}
In \cite{Morton}, the notion of direction is defined on components of motifs and  only for motifs  that do not contain chain link compounds and polycatenane compounds. It is explained that one can assign to each component of a DP diagram (resp. a flat motif) an {\it axial type} (see also \cite{Grishanov.part2}). More precisely, let $u$ and $v$ be two independent elements associated to a DP diagram $d_{\infty}$ and let $G$ be the discrete group of invariant translations generated by  $u$ and $v$. Then, any component of $d_{\infty}$ is invariant under some subgroup of $G$, which is either trivial for a closed curve, or is infinite cyclic if generated by an element $w \in G$. The direction of the component refers thus to the axis along which it runs.
\end{remark}

\subsection{Axis-motifs}\label{sec:axis-motif}

Recalling Definition~\ref{def:direction-element}, it is clear that the direction of an element depends, in most cases, only on its homotopy type or the homotopy type of its ribbon. In this subsection we present a {\it  geometrical representation of the direction of a motif}, that we call an {\it axis-motif}, by replacing elements of the motif by null-homotopic closed curves or essential closed curves. To construct an axis-motif corresponding to a given motif, we first define the {\it axis} of an element:

\begin{definition}\label{def:axis-element}
 The {\it axis of an element}, as a (set of) closed curves in $T^2$, is defined as follows.  Let $e$ be an element of a (flat) motif $\tau$. Then:

 \begin{enumerate}
   \item[a.] if $e$ has direction $(0,0)$, then the axis of $e$ is a trivial knot,  preserving possible intersections with longitude-meridian and their orders.
   \smallbreak  
   
   \item[b.] if $e$ has direction $(a,b)$, then the axis of $e$ is an $(a,b)$-torus knot.
   \smallbreak  
   
   \item[c.] if $e$ has direction $(\infty,\infty)$, then the axis of $e$ is a flat link of as many trivial components as the  components of $e$, preserving the intersections with longitude-meridian and their orders.
  \end{enumerate}  
\end{definition} 
\noindent Note that the difference between cases (a) and (c) is that, up to isotopy, in case (a) there is always a flat  motif for the DP tangle which encloses $e$ in its interior, while this is not true for case (c). 
Some examples are given after the next two definitions. 

\smallbreak

The notion of axis can also be generalized to each interlinked compound of a motif.

\begin{definition}\label{def:axis-compound}
  The {\it axis of an  interlinked compound} is defined as follows.  Let $\mu$ be an interlinked compound of a (flat) motif $\tau$. Then: 
  
  \begin{enumerate}
   \item[a.]  if $\mu$ is a null-homotopic compound of $\tau$ defined as a link of $c$ knots, then the axis of $\mu$ is a trivial link of $c$ components.
   \smallbreak 
   
   \item[b.] if $\mu$ is a ribbon compound of $\tau$ defined as a link of $k$ elements with direction $(a,b)$ and $c$ isolated knots, where $c = 0,1,\ldots$, then the axis of $\mu$ consists of a $(ka,kb)$-torus link and of a trivial link of $c$ components.
   \smallbreak  
   
   \item[c.] if $\mu$ is a cover compound of $\tau$, then the axis of $\mu$ is the set of  the axes of all its elements. 
  \end{enumerate} 
\end{definition}
\noindent Hence, the axis of an element or an interlinked compound realizes geometrically its direction. 

\smallbreak
We can now define the notion of axis-motif of a motif as follows: 

\begin{definition}\label{def:axis-motif}
    The {\it axis-motif} $\alpha(\tau)$ {\it of a (flat) motif } $\tau$ is defined as the projection of a  (flat) motif formed by replacing each element of $\tau$  by its axis,  and the axes are taken up to homotopy preserving the  orders of the intersections with longitude-meridian.  An axis motif that corresponds to a minimal motif shall be called {\it minimal axis-motif of } $\tau$.
\end{definition}

In other words, the axis-motif of a motif is the topological blueprint of its direction. Topological, since each direction is taken up to homotopy (especially in the case of flat motifs), and blueprint  since the direction of every element is represented in the axis-motif. In this context, a minimal axis-motif is a minimal blueprint of $\tau$.

\begin{examples}\label{ex:fig-DP tangle}
In Figure~\ref{slope-curves} we present examples of directions and axis-motifs corresponding to motifs of DP tangles.

\begin{itemize} 
   \item[-] Motif (a) contains two isolated knots and three essential components of the same direction. So, its direction is the set $\{(0,0), (1,0)\}$ of 2 distinct elements, and its corresponding axis-motif (a$'$) contains two trivial knots and a $(3,0)$-torus link.
   \smallbreak
   
   \item[-] Motif (b) contains three disjoint ribbon compounds; the  leftmost consists of two essential components, the middle one contains three essential component and the rightmost is a chain-link. The three compounds have the same direction, thus the motif has direction the singleton $\{(0,1)\}$, and its corresponding axis-motif (b$'$) is a $(0,6)$-torus link. 
   \smallbreak
   
   \item[-] Motif (c) contains a null-essential ribbon compound made of two essential components and three isolated knots. So, it has direction $\{(0,0), (2,1)\}$ and its corresponding axis-motif (c$'$) contains a $(4,2)$-torus link and three trivial knots.
   \smallbreak
   
   \item[-] Motif (d) contains an essential cover compound made of three essential components. It has direction $\{(1,0), (1,2), (-1,2)\}$ and its corresponding axis-motif (d$'$) contains a $(1,0)$-torus knot, a $(1,2)$-torus knot and a $(-1,2)$-torus knot.
   \smallbreak
   
   \item[-] Motif (e) contains a chain-polycatenane cover compound made of two interlinked chain-links, and a null-homotopic compound consisting of a single knot. It has direction $\{(0,0), (1,0), (0,1)\}$ and its corresponding axis-motif (e$'$) contains a $(1,0)$-torus knot, a $(0,1)$-torus knot and a trivial knot.
   \smallbreak
   
   \item[-] Motif (f) contains a full-polycatenane cover compound made of one knot. It has direction $\{(\infty,\infty)\}$ and its corresponding axis-motif (f$'$) contains one non-contractible  trivial knot.
   \smallbreak 
   
   \item[-] Motif (g) contains a full-polycatenane cover compound made of fourteen components. It has direction $\{(\infty,\infty)\}$ and its corresponding axis-motif (g$'$) consists of fourteen  trivial knots, preserving the longitude/meridian intersections.
   \smallbreak
   
   \item[-] Motif (h) contains a cover compound consisting of an interlinkage of: a full-polycatenane of one knot, a chain-link, three essential curves (pink, green, black), and a trivial knot. It also contains two null-homotopic compounds of one and two components respectively. It has direction $\{(\infty,\infty), (0,0), (1,0), (0,1)\}$ and its corresponding axis-motif (h$'$) contains a $(2,0)$-torus link, a $(0,2)$-torus link, four trivial knots, and one non-contractible  trivial knot.
\end{itemize}
\end{examples}

\begin{figure}[H]
\begin{center}
\includegraphics[width=5in]{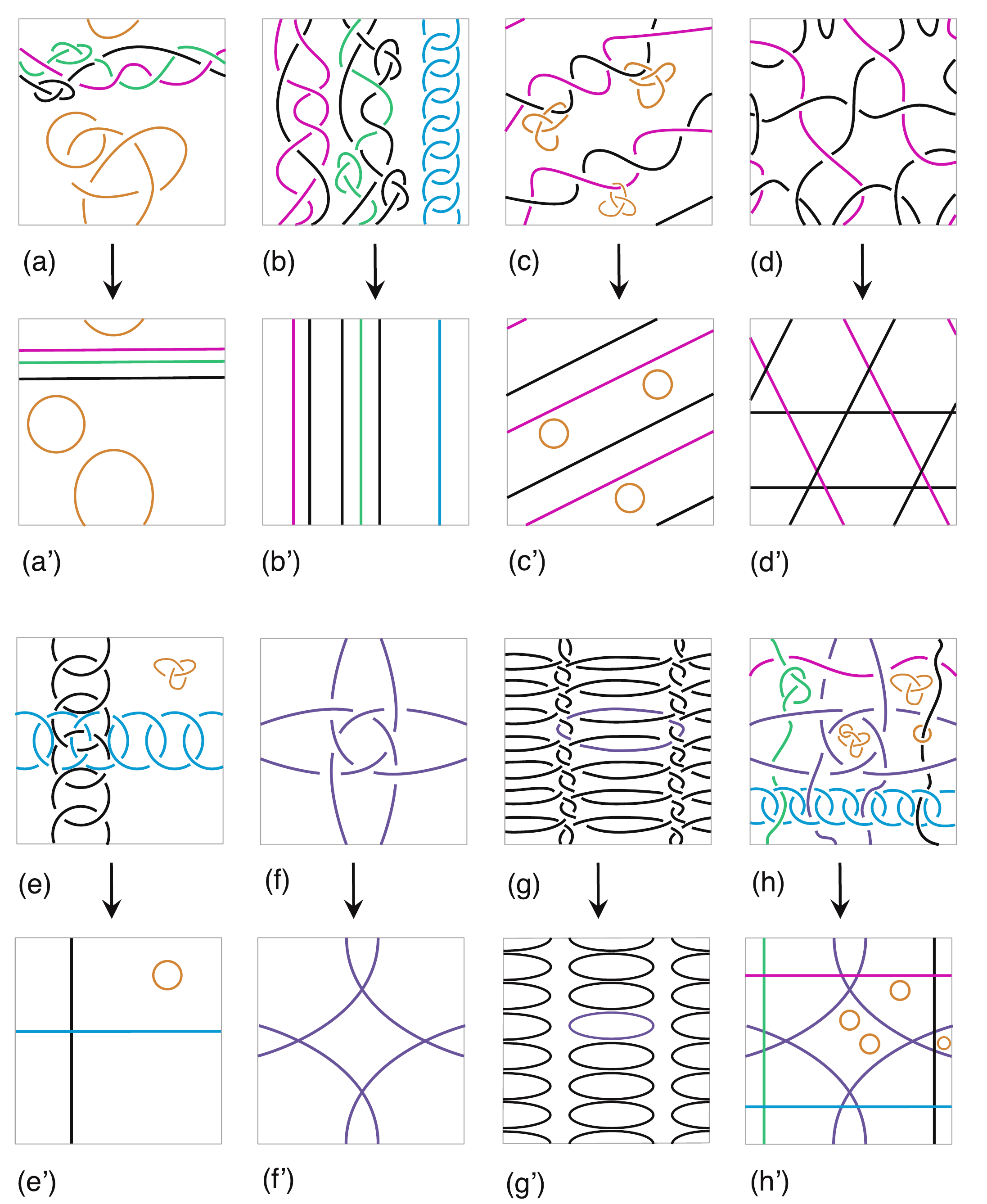}
\end{center}
\caption{\label{slope-curves} The corresponding  axis-motifs of the flat motifs in Figure~\ref{interlinked-compounds}. Color coding helps identify elements in the motif and their corresponding axes in the axis motif.}
\end{figure}

\subsection{Directional types of DP tangles} \label{sec:directional-type}

In this subsection we discuss the invariance of the direction of a motif under equivalence of its DP tangle, as characterized by Theorem~\ref{th:equivalence}. 

\smallbreak
Since the notion of direction has been defined considering homotopy of the components and their intersections with the torus longitude and meridian, we can distinguish the following cases. First, it is immediate that the direction of a motif is invariant under shifts, motif isotopies, and re-scaling transformations. Regarding Dehn twists, it is clear that they preserve the direction of an isolated knot or of a full polycatenane. However, the direction of an element with direction $(a,b)$  is not preserved under Dehn twists. Finally, it is also clear that scale equivalence preserves the direction of an isolated knot or of a full polycatenane. However, the direction of an element with direction $(a,b)$ is not preserved under scale equivalence (as for example in Figure~\ref{scale-direction}~(c)), except  for the case where an element has the same direction of the torus longitude or the torus meridian, as illustrated in Figure~\ref{scale-direction}~(a) and~(b). 

\begin{figure}[H]
\begin{center}
\includegraphics[width=5in]{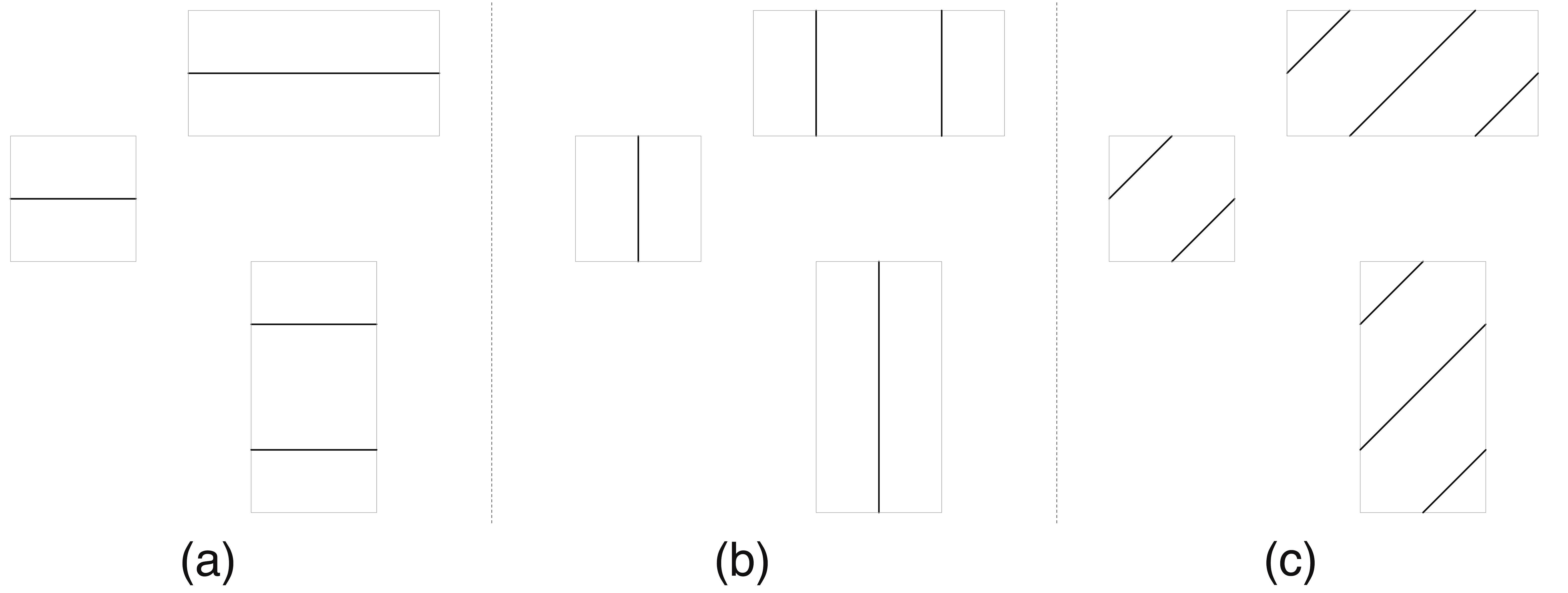}
\end{center}
\caption{\label{scale-direction} Three flat motifs and their double covers: (a) the three motifs have direction $\{(1,0)\}$; (b) the three motifs have direction $\{(0,1)\}$; (c) the motif on the left has direction $\{(1,1)\}$, the motif on the top has direction $\{(2,1)\}$, the motif on the bottom has direction $\{(1,2)\}$.}
\end{figure}

In conclusion, {\it the direction of a motif is not an invariant of its corresponding DP tangle}, if it contains essential components and chain-links. Yet, as observed in \cite{Grishanov.part2} for the case of minimal motifs consisting only of essential components, {\it the total number of distinct directions} of a motif resp. DP tangle is a topological invariant of the DP tangle. Indeed, this number is invariant under Dehn twists, since, for example, all torus knots are Dehn-equivalent, and analogously for torus links (recall that all torus knots and torus links, lift, by shearing transformation, to isotopic DP tangles \cite{DLM}). Moreover, by restricting to minimal motifs, scale-equivalence does not apply.

\smallbreak
Here we generalize the above result of \cite{Grishanov.part2} to any class of motifs and we do not restrict  to only minimal ones. Namely, we have  the following result. 

\begin{theorem} \label{th:distinct-directions} 
The total number of distinct directions of a (flat) motif resp. DP tangle is a topological invariant of the DP tangle.
\end{theorem}  

\begin{proof}
Indeed, let  $\tau$ be a (flat) motif in $T^2 \times I$  of any class and let $\tau_{\infty}$ be the corresponding DP tangle in $\mathbb{E}^2 \times I$. Theorem~\ref{th:interlinked-compounds} ensures that the class and subclass of   $\tau$  will remain invariant under DP tangle equivalence (cf. Theorem~\ref{th:equivalence}).  Therefore, by the Definitions~\ref{def:direction-element},~\ref{def:direction-compound},and~\ref{def:direction}, which apply to motifs of any class and subclass, the total number of distinct directions of   $\tau$ resp. $\tau_{\infty}$  is clearly invariant under local isotopy moves, shift equivalence, re-scaling transformations, and orientation preserving affine transformations. 
  
Regarding, specifically, Dehn twists, the same arguments as in \cite{Grishanov.part2} apply for motifs of any class and subclass. Namely a Dehn twist will preserve the directions $(0,0)$ and $(\infty,\infty)$ and will change a direction $(a,b)$ to a direction $(a',b')$, but not to $(0,0)$ or $(\infty,\infty)$. So, the number of distinct directions is preserved. 

Scale equivalence, finally, will preserve the directions $(0,0)$ and $(\infty,\infty)$ and may change a direction $(a,b)$ to a direction $(a',b')$, but not to $(0,0)$ or $(\infty,\infty)$, as discussed above (see Figure~\ref{scale-direction}). Moreover, since the direction of a motif is defined as a set, even if a motif that is a multiple of $\tau$ may contain more components of the same direction, its cardinality remains unchanged. The   proof of the theorem is now concluded.
\end{proof}

As examples to illustrate Theorem~\ref{th:distinct-directions}, observe that the number of directions of the DP tangle illustrated in  Figures~\ref{shearing} and~\ref{crossing} remains equal to two after scale equivalence, Dehn twists and shifts. Similarly, the number of directions  of the DP tangle of Figure~\ref{Tknot-Tlink} (that is, one) is invariant under re-scaling transformations and scale equivalence.

\begin{note}
   If the number of distinct directions is at least two, then the motif (resp. the DP tangle) is of class cover. However, for a direction of cardinality one, we cannot distinguish the different classes of motif (resp. of DP tangle).  
\end{note}

According to Definition~\ref{def:direction}, DP tangles can be distinguished into different {\it directional types}, for which we introduce the following notation:

\begin{definition}  \label{def:DP type} 
Let $\tau$ be a motif in $T^2 \times I$ and $\tau_{\infty}$ in $\mathbb{E}^2 \times I$. 
\begin{itemize} 
   \item [-] if $\tau$ has direction $\{(0,0)\}$, then $\tau$ and $\tau_{\infty}$ are said to be of {\it type $0$} or a {\it $0$-motif} and a {\it $0$-DP tangle}, respectively. 
   \smallbreak
   
   \item [-] if $\tau$ has direction $\{(\infty,\infty)\}$, then $\tau$ and $\tau_{\infty}$ are said to be of {\it type $\infty$} or a {\it $\infty$-motif} and a {\it $\infty$-DP tangle}, respectively. 
   \smallbreak   
   
   \item [-] if $\tau$ has a set of direction of cardinality $N$ that does not contains $(\infty,\infty)$ nor $(0,0)$, then $\tau$ and $\tau_{\infty}$ is said to be of {\it type $N$} or a {\it $N$-motif} and a {\it $N$-DP tangle}, respectively. 
   \smallbreak
   
   \item [-] if $\tau$ has direction $\{(\infty,\infty), (0,0)\}$, then $\tau$ and $\tau_{\infty}$ are said to be of {\it type $(0,\infty)$} or a {\it $(0,\infty)$-motif} and a {\it $(0,\infty)$-DP tangle}, respectively.
   \smallbreak
   
   \item [-] if $\tau$ has a set of direction of cardinality $N+1$, containing $(0,0)$ but not $(\infty,\infty)$, then $\tau$ and $\tau_{\infty}$ is said to be of {\it type $(N,0)$} or a {\it $(N,0)$-motif} and a {\it $(N,0)$-DP tangle}, respectively. 
   \smallbreak
   
   \item [-] if $\tau$ has a set of direction of cardinality $N+1$, containing $(\infty,\infty)$ but not $(0,0)$, then $\tau$ and $\tau_{\infty}$ is said to be of {\it type $(N,\infty)$} or a {\it $(N,\infty)$-motif} and a {\it $(N,\infty)$-DP tangle}, respectively. 
   \smallbreak
   
   \item [-] if $\tau$ has a set of direction of cardinality $N+2$, containing $(\infty,\infty)$ and $(0,0)$, then $\tau$ and $\tau_{\infty}$ is said to be of {\it type $(N, 0,\infty)$} or a {\it $(N,\infty,0)$-motif} and a {\it $(N,0, \infty)$-DP tangle}, respectively. 
   \end{itemize}
\end{definition}

Some examples of types of motifs are presented in Figure~\ref{DPs}. We can now state the following result, which is a refinement of Theorem~\ref{th:distinct-directions}.

\begin{theorem} \label{th:directional types} 
The  directional type of a DP tangle is a topological invariant of the DP tangle.
\end{theorem} 

\begin{proof}
The proof is a direct consequence of Theorem~\ref{th:distinct-directions} and Definition~\ref{def:DP type}.
\end{proof}

\begin{remark}
The number of distinct directions  of a (flat) motif  $\tau$ is the same as the number of distinct directions  of its axis-motif $\alpha(\tau)$. We observe that the axis-motif  facilitates the identification of the directions and, consequently, of the number of distinct directions of the DP tangle $\tau_\infty$, which is usually not trivial, as one can confirm in the examples of Figure~\ref{scale-direction}. In fact, the axis-motif captures more that just the number of distinct directions or even the directional type of $\tau$:  it also depicts {\it the number of elements of the same direction}. We note further that the axis-motif is  sensitive to changes of the number of elements in scale equivalence, as well as to changes of axes of elements under Dehn twists. In the next section we shall exploit further these observations.
\end{remark}

\begin{figure}[H]
\begin{center}
\includegraphics[width=4.7in]{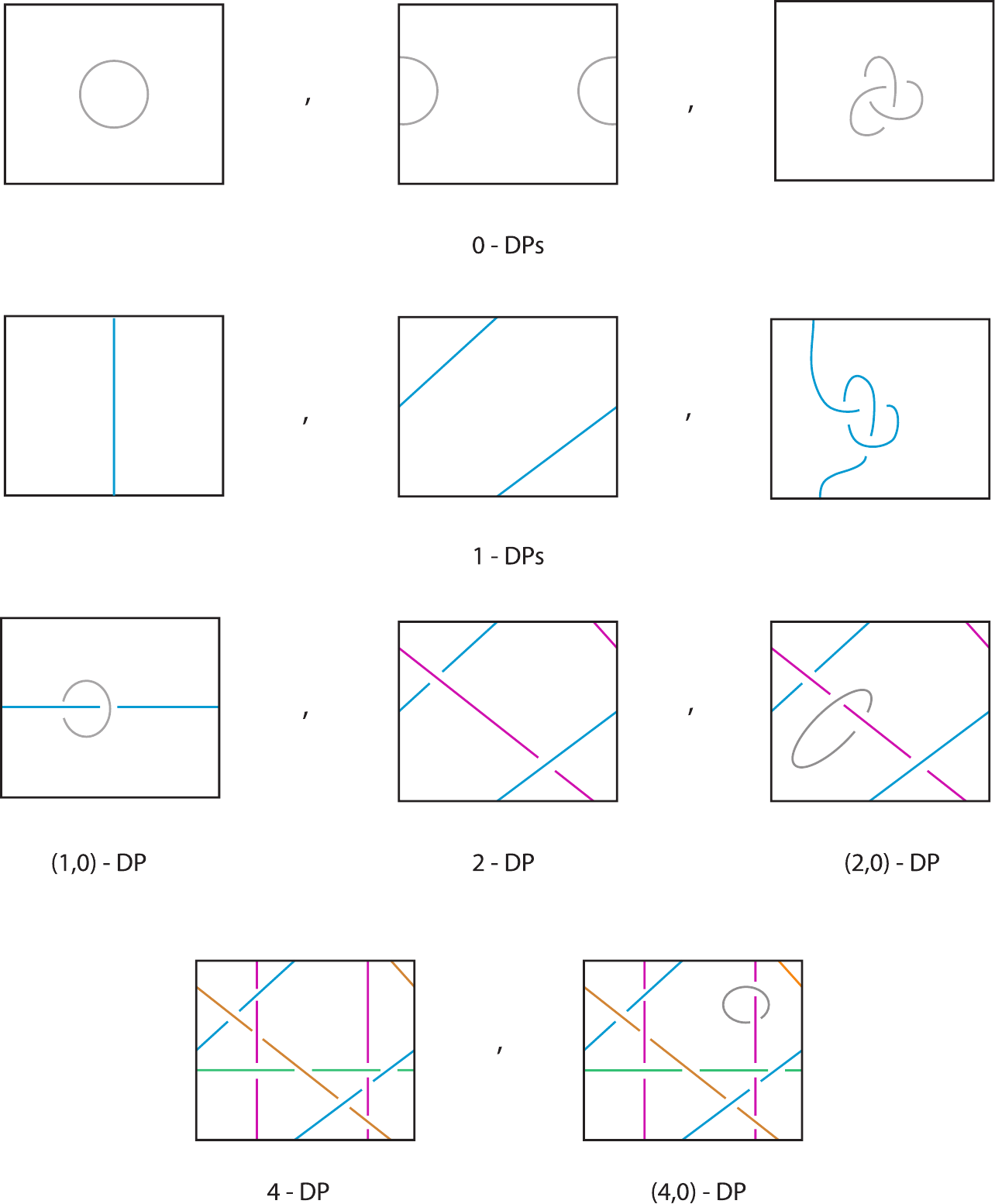}
\end{center}
\caption{\label{DPs} Different directional types of motifs. First row: three 0-DP motifs. Second row: three 1-DP motifs. Third row: a (1,0)-DP motif, a 2-DP motif and a (2,0)-DP motif. Fourth row: a 4-DP motif and a (4,0)-DP motif. Different colors indicate different directions.}
\end{figure}

\section{Comparing invariants and computing examples}\label{sec:tables}

The directional invariants defined herein, namely the class and subclasses, the directional type and the axis-motif, serve as a valuable tool in constructing a classification table for DP tangles. It is worth noting that each one of these invariants alone is insufficient for classifying the topological complexity of a DP-tangle. Besides, recall that many numerical invariants have been introduced in \cite{Grishanov.part2} for textiles (weaves and knits), which can add to the classification efforts. We single out three of them:
\begin{enumerate}
    \item [-] the {\it crossing number}, defined as the minimal number of crossings in a motif of a DP tangle $\tau_{\infty}$, among all possible minimal motifs of $\tau_{\infty}$.
    \smallbreak
    \item [-] the {\it number of components}, defined as the number of component in any minimal motif of a DP tangle.
    \smallbreak
    \item [-] the {\it axis number}, which refers to the number of directions in this paper.
\end{enumerate}

\begin{remark}
  The above definitions of {\it crossing number} and {\it number of components} clearly extend to any DP tangle (not just textiles) and to the set of all motifs, not only minimal. Then they are clearly DP tangle invariants in the general case,  since they remain invariant under the equivalence relation stated in  Theorem~\ref{th:equivalence}. For the {\it axis number} 
  we already mentioned above how they relate to our {\it number of directions} and {\it directional type}, only here we define these invariants for all possible motifs, not just fabric-related. 
\end{remark}

In Tables~\ref{table1} and~\ref{table2} we present several examples of minimal motifs, where we compute the number of directions, the directional type, the class and the minimal axis-motif for each DP tangle, comparing at the same time with their crossing number and the number of components. For comparison reasons, many of the motifs studied in \cite{Grishanov.part2} appear in Tables~\ref{table1} and~\ref{table2}.

Our computations show that the invariants introduced in the present work are distinct from crossing number and number of components. Indeed, they can distinguish some motifs of Tables~\ref{table1} and~\ref{table2} as follows. 

\begin{itemize}
    \item [*] motif 3 differs from motifs 1 and 2 only by its class,
    \smallbreak
    
    \item [*] motif 4 differs from motifs 5 and 6 by its number of directions, its directional type and its class,
    \smallbreak
    
    \item [*] motifs 7 and 8 differ by their number of directions, their directional type and their class,
    \smallbreak
    
    \item [*] motif 11 differs from motifs 9 and 10 by its directional type and its set of subclasses,
    \smallbreak
    
    \item [*] motifs 12 and 13 differ by their directional type and their set of subclasses,
    \smallbreak
    
    \item [*] motifs 13 and 14 only differ by their class,
    \smallbreak
    
    \item [*] motif 15 differs from motifs 13 and 14 by its directional type and its class,
    \smallbreak
    
    \item [*] motifs 16 and 17 differ by their number of directions, their directional type and their set of subclasses,
    \smallbreak
    
    \item [*] motif 18 differs from motifs 19 and 20 by its number of directions, its directional type and its class and set of subclasses,
    \smallbreak
    
    \item [*] motifs 19 and 20 only differ by their set of subclasses.
\end{itemize}

Finally, it is also worth mentioning that even the combination of all these invariants cannot distinguish some motifs, as for example motifs 1 and 2, motifs 5 and 6, or motifs 9 and 10.


\begin{table}[H]
\centering
   \includegraphics[width=5.3in]{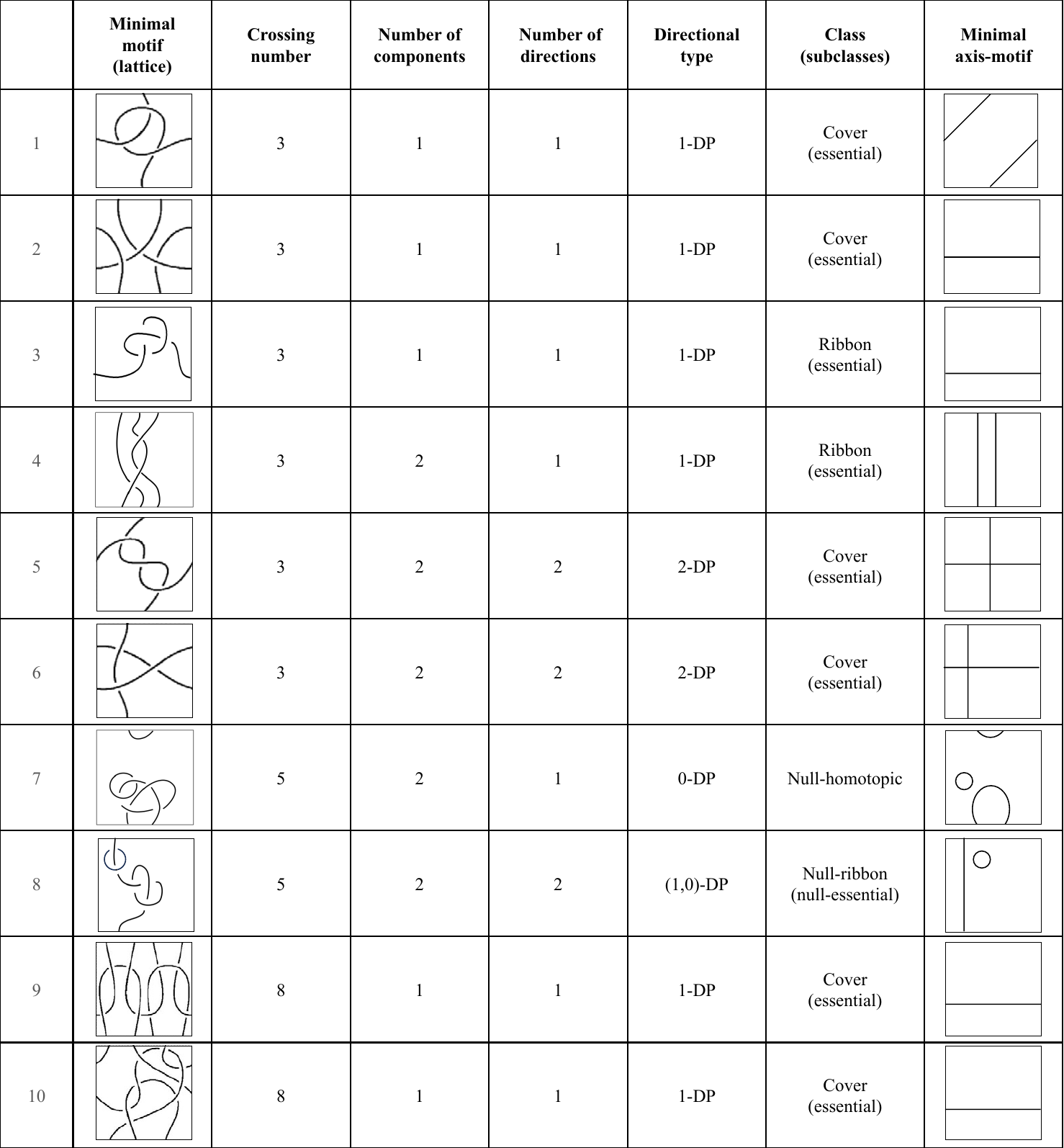}
      \caption{\label{table1} Minimal motifs and some  invariants of their topological properties.}
\end{table}
\unskip

\begin{table}[H]
\centering
   \includegraphics[width=5.3in]{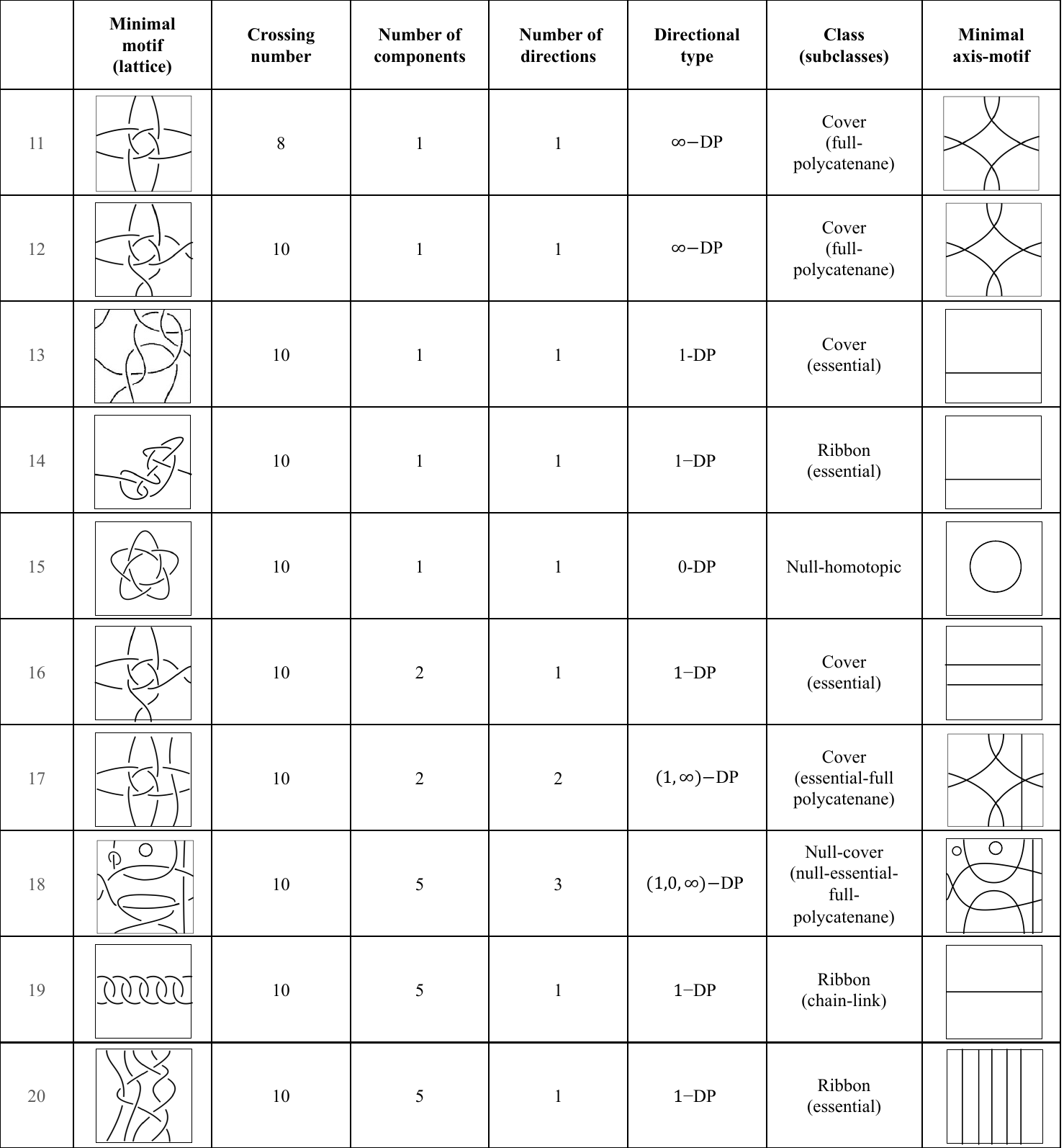}
      \caption{\label{table2}  Minimal motifs and some  invariants of their topological properties.}
\end{table}
\unskip


\noindent {\bf Concluding remarks.} The invariants of DP tangles constructed in this work are measures of their topological complexity. All of them refer to global topological properties of a DP tangle, and they add to the list of the existing invariants. In a forthcoming paper \cite{DLM2} we shall present a more robust invariant that encompasses the directional type of a DP tangle, as defined in Definition~\ref{def:DP type}. This enhanced invariant gives rise to a canonical embedding, and also aids toward detecting a minimal motif for a DP tangle.


\end{document}